\documentclass{amsart}
\usepackage{CJKutf8}
\usepackage{amsmath,amsfonts,amssymb,amscd,bm,latexsym}
\usepackage[colorlinks=true]{hyperref}
\usepackage{tikz}

\usepackage{pgflibraryarrows}
\usepackage{pgflibrarysnakes}
\usepackage{pgflibraryarrows}
\usepackage{pgflibrarysnakes}
\hypersetup{urlcolor=blue, citecolor=blue}
\numberwithin{equation}{section} 

\def\<{\langle}             \def\>{\rangle}

\newtheorem{thm}{Theorem}[section]
\newtheorem{coro}[thm]{Corollary}

\newtheorem{lem}[thm]{Lemma}
\newtheorem{proposition}[thm]{Proposition}
\newtheorem{prop}[thm]{Proposition}

\theoremstyle{definition}

\newtheorem{rem}{Remark}[section]

\newcommand{\beeq}{\begin{equation}}\newcommand{\eneq}{\end{equation}}
\newcommand{\al}{\alpha}    \newcommand{\be}{\beta}
\newcommand{\de}{\delta}

\newcommand{\om}{\omega}    
    
\newcommand{\R}{\mathbb{R}}\newcommand{\Z}{\mathbb{Z}}

\newcommand{\ep}{\varepsilon}
\newenvironment{prf}{\noindent {\bf Proof.} }{\endprf\par}
\def \endprf{\hfill  {\vrule height6pt width6pt depth0pt}\medskip}
\newcommand{\pt}{\partial_t}\newcommand{\pa}{\partial}
\newcommand{\les}{{\lesssim}}
\newcommand{\supp}{\,\mathop{\!\mathrm{supp}}}

\newcommand{\gm}{\mathfrak{g}} 
 
\newcommand{\la}{\lambda}

\numberwithin{equation}{section}
\title[Semilinear wave equations with mixed nonlinearities]
      {Blow up for small-amplitude semilinear wave equations with mixed nonlinearities on asymptotically Euclidean manifolds}
\author{Mengyun Liu}
\address{Department of Mathematics\\                	
Zhejiang Sci-Tech University\\                Hangzhou 310018, P. R. China}
\email{mengyunliu@zstu.edu.cn}

\author{Chengbo Wang}\address{School of Mathematical Sciences\\                Zhejiang University\\                Hangzhou 310027, P. R. China}\email{wangcbo@zju.edu.cn}
\urladdr{http://www.math.zju.edu.cn/wang}

\thanks{The authors were supported by NSFC 11971428.}
\date{\today}
\dedicatory{} \commby{}
\begin{document}
\begin{abstract}
In this work, we investigate the problem of finite time blow up as well as the upper bound estimates of lifespan for solutions to small-amplitude semilinear wave equations with mixed nonlinearities $c_1 |u_t|^p+c_2 |u|^q$, posed on asymptotically Euclidean manifolds,
 which 
is related to both the Strauss conjecture and Glassey conjecture. In some cases, we obtain existence results, where the lower bound of the lifespan agrees with the upper bound in order.
In addition, 
our results apply for semilinear damped wave equations,
when the coefficient of the dissipation term is integrable (without sign condition) and space-independent.
 \end{abstract}



\maketitle

\section{Introduction}

Let $(\R^{n}, \gm)$ be a asymptotically Euclidean (Riemannian) manifold, with $n\ge 2$. 
By asymptotically Euclidean, we mean that
$(\R^n, \gm)$ is certain perturbation of the Euclidean space $(\R^n, \gm_0)$. More precisely, we assume
$\gm$ can be decomposed as 
\begin{align}
\label{dl1}
\gm=\gm_{1}+\gm_{2}\ ,
\end{align}
where $\gm_{1}$ is a spherical symmetric, long range perturbation of $\gm_0$, and $\gm_{2}$ is a short range perturbation.
By definition, there exists polar coordinates $(r,\omega)$ for $(\R^n,\gm_1)$, in which we can write
\beeq
\label{dl3}
\gm_{1}=K^{2}(r)dr^{2}+r^{2}d\omega^{2}\ , 
\eneq
where $d\om^2$ is the standard metric on the unit sphere $\mathbb{S}^{n-1}$, and
\beeq
\label{dl2}
|\pa^{k}_{r}(K-1)|\les  \<r\>^{-k-\rho}
,  k=0, 1, 2, 
\eneq
for some given positive constant $\rho$. Here, 
$\langle x\rangle=\sqrt{1+|x|^2}$, and
we use $A\les B$ to stand for $A\leq CB$ for some constant $C>0$. 
Equipped with the coordinates $x=r\omega$, we have
$$\gm=g_{jk}(x)dx^j dx^k\equiv \sum^{n}_{j,k=1}g_{jk}(x)dx^j dx^k\ ,\ \gm_2=g_{2, jk}(x)dx^j dx^k\ ,$$
where we have used the convention that Latin indices $j$, $k$ range from $1$ to $n$ and the Einstein summation convention for repeated indices. 
Concerning $\gm_2$, we assume
it satisfies
\beeq
\label{eq-g2-ae2}
\nabla^\be g_{2,jk}=\mathcal{O}(\<r\>^{-\rho-1-|\be|}), |\be|\le 2\ .
\eneq
By these assumptions, it is clear that 
there exists a constant $\delta_{0}\in(0, 1)$ such that 
\beeq
\label{unelp}
\delta_{0}|\xi|^2\le
g^{jk}(x)\xi_{j}\xi_{k}\le
\delta_{0}^{-1} |\xi|^2, \forall \ x, \xi\in\R^{n},\ 
K\in (\delta_{0}, 1/\delta_{0})
\ ,
\eneq
where
$(g^{jk}(x))$ denotes the inverse of $(g_{jk}(x))$.

In this paper, we are interested in the investigation of the blow up part of the analogs of the Glassey conjecture and related problems on asymptotically Euclidean manifolds.
More precisely, we will study the blow up of solutions for 
the following semilinear wave equations with small initial data,
  posed on asypmtotically Euclidean manifolds \eqref{dl1}-\eqref{eq-g2-ae2}
\beeq
\label{nlw}
\begin{cases}
\pa^{2}_{t}u-\Delta_{\gm}u=c_1 |u_{t}|^{p}+ c_2 |u|^q\\
u(0,x)=\ep u_0(x), u_{t}(0,x)=\ep u_1(x)\ ,
\end{cases}
\eneq
where, 
$\Delta_{\gm} =\nabla^j\partial_{j}$ is the standard Laplace-Beltrami operator,  $p, q>1$, $c_1, c_2\ge 0$ and $\ep>0$ is a small parameter. Concerning the initial data, we assume 
\beeq
\label{hs2}
0\leq u_0, u_1 \in C^{\infty}_{0}(\R^{n}),\quad \supp(u_0, u_1)\subset\{x\in \R^{n}; |x|\leq R_{0}\}\ , 
\eneq
 for some $R_{0}>0$.

 Before stating our results, let us briefly review the history of this problem. When $c_1=0<c_2, \gm=\gm_0$, the problem is related to the Strauss conjecture, where the critical power is given by the Strauss exponent $p_S(n)$, where $p_S$ is the positive root of the equation:
$$(n-1)q^{2}-(n+1)q-2=0
\ .$$
When there is no global solutions, and initial data are nontrivial and nonnegative, it has been proved that the upper bound of the lifespan is
 \beeq\label{eq-life}
T_\ep\leq S_{\ep}:=\left\{
\begin{array}{ll }
  C_0 \ep^{\frac{2q(q-1)}{(n-1)q^{2} - (n+1)q - 2}}    & 1<q<p_S(n);   \\
\exp(C_0\ep^{-q(q-1)})      &   q=p_S(n).
\end{array}\right.
\eneq
In addition, when $n=2$, $1<q<2$ and
$u_1\neq 0$,  the upper bound of the lifespan can be improved to 
\beeq
\label{eq-life-special}
T_\ep\leq S_\ep^1= C_0 \ep^{-\frac{q-1}{3-q}}, \ n=2, 1<q<2\ ,
 \eneq
for some constant $C_0>0$. We refer
\cite{LW2019} for discussion of the history and related results.
 
 When $\gm = \gm_{0}$ and $c_1\neq 0=c_2$, the problem of determining the sharp range of powers of $p>1$ for global existence versus blow up for arbitrary small initial data, is known as the Glassey conjecture, where the critical power $p$ for \eqref{nlw}
is conjectured to be
$$p_G(n):=1 + \frac{2}{n-1}\ .$$
In the important particular case of $p=2$, the problem with $n=3$ has been well known to be critical, which admits almost global solutions in general, with the estimates of the lifespan, denoted by $T_\ep$, given by $\ln T_\ep\sim \ep^{-1}$, see
John \cite{FJ1981} and
John-Klainerman \cite{JK84}.
In general, the global existence with $p>p_G(n)$ is known for $n=2,3$ through the works of
Sideris \cite{sideris},  Hidano-Tsutaya \cite{hk} and Tzvetkov \cite{nt}, while 
nonexistence of global solutions for $1<p\le p_G(n)$ and any $n\ge 2$ as well as the upper bound of lifespan,
\beeq
\label{main-es}
T_{\ep}\leq G_{\ep}:=
\begin{cases}
C_{0}\ep^{-\frac{2(p-1)}{2-(n-1)(p-1)}}, \ 1<p<p_{G}\ ,\\
\exp(C_{0}\ep^{-(p-1)}), \ p=p_{G}\ ,
\end{cases}
\eneq
has also been well-known (at least for the case $u_1\neq 0$), see Zhou \cite{zhou} and references therein.
The high dimensional existence part remains open, except for the radial case,
see Hidano-Wang-Yokoyama \cite{hyw}. 
The upper bound \eqref{main-es} is also known to be sharp in general, at least for radial data, see Hidano-Jiang-Wang-Lee \cite{HJWLax} and references therein.

 When $\gm=\gm_0$ and $c_1c_2\neq 0$, the problem
 is related to both the Glassey conjecture and the Strauss conjecture. It turns out that a new critical curve occur in the expected region of global existence $p>p_G(n)$, and $q>p_S(n)$, that is
 $$q>p_{S}, p>p_G, (q-1)((n-1)p-2)=4\ ,$$
 where the critical and super-critical case is known to admit global existence, at least for $n=2, 3$, see Hidano-Wang-Yokoyama \cite{HWY}, while there is non-existence of global existence for the sub-critical case, as well as 
an upper bound of the lifespan 
  \beeq
 \label{mixed-lifespan1}
  T^{p,q}_{\ep}\leq Z_{\ep}:=C_0 \ep^{-\frac{2p(q-1)}{2q+2-(n-1)p(q-1)}},
  \eneq
  in the region
\beeq
\label{mixed-region1} (q-1)((n-1)p-2)<4,
   p\leq\frac{2n}{n-1},  q< \frac{2n}{n-2}\ , \eneq
   when $u_1\neq 0$,
  see Han-Zhou \cite{HZ2014}. Note that,  in the recent work of Lai-Takamura \cite{LTmix}, they remove the restriction for $p$ and obtain the upper bound when $p>\frac{2n}{n-1}$ and
$u_1\neq 0$: 
\beeq
 \label{mixed-lifespan2}
  T^{p,q}_{\ep}\leq C_0 \ep^{-\frac{(q-1)}{q+1-n(q-1)}}, p>\frac{2n}{n-1},  q< \frac{n+1}{n-1}\ .
 \eneq
In fact, in this case, since we have $q<\frac{n+1}{n-1}<p_{S}$, we could compare the upper bound with $S_\ep$ in \eqref{eq-life} and find that \eqref{mixed-lifespan2} improves the upper bound exactly when $n=2$, $1<q<2$, that is \eqref{eq-life-special}. 
Moreover, we find that the additional restriction for $q$ in \eqref{mixed-region1}, that is, $q<\frac{2n}{n-2}$ is also not necessary, except $q<1+\frac{4}{n-3}$ implied by the first inequality, by recasting the proof. 

Notice that the results
 \eqref{eq-life},
 \eqref{eq-life-special} and
 \eqref{main-es} apply also in the case of $c_1c_2\neq 0$, by comparing $S_\ep$, $G_\ep$ with $Z_\ep$, when $u_1\neq 0$, we have 
  \beeq
 \label{zonghelife}
 T^{p,q}_{\ep}\les
 \begin{cases}
 S_{\ep},\ \ \  q\leq p_S, p\geq q\\
 Z_{\ep}, \ \ \ p\leq q\leq 2p-1, (q-1)((n-1)p-2)<4\\
 G_\ep, \ \ \ p\leq p_G, \ q\geq 2p-1\\
   S^1_{\ep},\ \ \  q\in (1,2), p\ge 2\frac{q+1}{5-q}, n=2
 \end{cases}
 \eneq
 See Figure \ref{tu} for the illustration of lifespan estimate of $T^{p,q}_\ep$ in blows up region. Concerning the sharpness of the upper bound of the lifespan, at least for $n=2,3, 4$, $q,p\ge 2$ and $q>2/(n-1)$, see
\cite{HWY},
Wang-Zhou \cite{WZ18} and
 Dai-Fang-Wang \cite{DFW}.

 \begin{figure}
\centering
\begin{tikzpicture} 
\draw[thick] (2.414,2.414)--(2.414,4) (3,2)--(5,2);
\draw[thick, dashed, domain=2.414:3] plot (\x, {{1+2/((\x)-1)}});
\node[right] at (2.6,2.3) {\tiny $(q-1)((n-1)p-2)=4$};
\node[right] at (1.5,2.5) {$S_{\ep}$};
\node[right] at (3,1.5) {$G_{\ep}$};
\draw[fill=black,line width=1pt] (2.414,2.414) circle[radius=0.3mm];
\draw[fill=black,line width=1pt] (3,2) circle[radius=0.3mm];
\node[below left] at (1,1) {\tiny $(1,1)$};
\node[left] at (1,2.414) {$p_S$};
\node[right] at (2,1.9) {$Z_{\ep}$};
\node[left] at (1,2) {$p_G$};
\node[below] at (2.414,1) {\small $p_S$};
\node[below] at (3.5,1) {\tiny $1+4/(n-1)$};
	\draw[thick] (1,1)--(2.414,2.414) (1,1)--(3,2);
	\draw[thick,-stealth] (1,1)--(1,4) node[left]{$p$};
	\draw[thick,-stealth] (1,1)--(5,1) node[below]{$q$};
\node[below] at (3,0.5) {\tiny $n\ge 3$};

\draw[thick] (8.414,2.414)--(8.414,4) (9,2)--(11,2);
\draw[thick, dashed, domain=8.414:9] plot (\x, {{1+2/((\x)-7)}});
\draw[thick](7,1) to[out=30,in=240] (7.5,1.5);
\node[right] at (8.6,2.3) {\tiny $(q-1)(p-2)=4$};
\node[right] at (7,2.5) {\small $S^1_{\ep}$};
\node[right] at (7.5,3.2) {$S_{\ep}$};
\node[right] at (9,1.5) {$G_{\ep}$};
\draw[fill=black,line width=1pt] (8.414,2.414) circle[radius=0.3mm];
\draw[fill=black,line width=1pt] (9,2) circle[radius=0.3mm];
\node[below left] at (7,1) {\tiny $(1,1)$};
\node[left] at (7,2.414) {$p_S$};
\node[right] at (8,1.9) {$Z_{\ep}$};
\node[left] at (7,2) {$p_G$};
\node[below] at (7.5,1) {\small $2$};
\node[below] at (8.414,1) {\small $p_S$};
\node[below] at (9,1) {\small $5$};
	\draw[thick] (7.5,1.5)--(8.414,2.414) (7,1)--(9,2) (7.5,1.5)--(7.5, 4);
	\draw[thick,-stealth] (7,1)--(7,4) node[left]{$p$};
	\draw[thick,-stealth] (7,1)--(11,1) node[below]{$q$};

\node[below] at (9,0.5) {\tiny $n=2$};
\end{tikzpicture}

 \caption{Upper bound of the lifespan $T^{p,q}_\ep$ when $u_1\neq 0$}
\label{tu}
\end{figure}
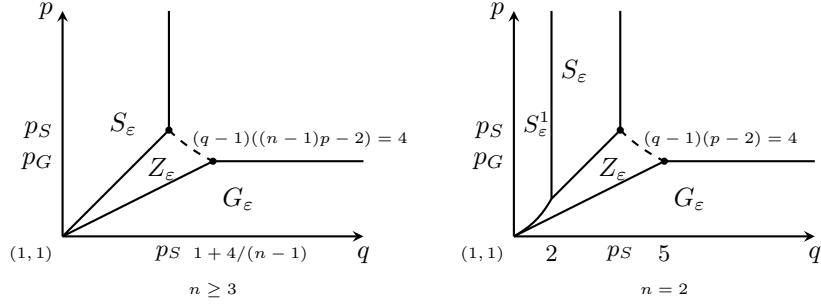

 In recent years, there have been many works concerning the analogs of the problem and related variants on asymptotically flat manifolds, including asymptotically Euclidean manifolds, black hole space-time.
 In general, we expect that the same critical power extends to asymptotically flat manifolds, at least when the manifolds enjoy certain geometric structure such as non-trapping.
 
 The existence part of the analogs of the Glassey conjecture for asymptotically Euclidean manifolds and  certain small space time perturbation of the flat metric 
  have been partly investigated by
 Sogge-Wang \cite{sw2010} ($n=3$, $p=p_G=2$) and
  the second author in \cite{W15, W15j}.
For exterior domain, see
Zhou-Han \cite{z-h2011} for blow up results and
the second author \cite{W15j} for existence with non-trapping obstacles. For the related works on black hole, Luk \cite{Luk} obtained the small data global existence on Kerr spacetime, when the nonlinear term satisfies the null condition.
Recently, Lai \cite{LLM} studied the blow up of Glassey conjecture on Schwarzschild spacetime for $1<p\leq 2=p_G(3)$.


A closely related problem to 
\eqref{nlw} is the following 
 \beeq
\label{nlw2}
\begin{cases}
\pa^{2}_{t}u-\Delta_{\gm}u+b(t)u_{t}=c_1 |u_{t}|^{p}+ c_2 |u|^q\\
u(0,x)=\ep u_0(x), u_{t}(0,x)=\ep u_1(x)\ ,
\end{cases}
\eneq where $b\in L^1$.
There have been some recent works concerning this problem, with typical damping coefficient $b(t)=\mu (1+t)^{-\be}$.
It turns out that the behavior of \eqref{nlw2} depends on $\mu$ and $\beta$. When $\beta >1$ and $\gm=\gm_0$, it is referred as ``scattering'' since the linear part of \eqref{nlw2} behaves like linear wave equations. In this case, when $c_1>c_2=0$,
Lai-Takamura \cite{LTgla} proved blow-up results as well as 
estimates of the lifespan \eqref{main-es},
when $\gm=\gm_0$, $1 <p \leq p_{G}$, $0\leq b(t)\in L^{1}$ and $u_1\neq 0$. Similarly, when $c_1c_2\neq 0$,
Lai-Takamura \cite{LTmix} obtained  \eqref{mixed-lifespan1}
for $\gm=\gm_0$, $u_1\neq 0$ and typical damping $b(t)=\mu (1+t)^{-\be}$, $\mu>0$, $\be>1$.

On the other hand,
the global existence for $p>p_{G}$ when $n=3$ and $n\ge 4$ with $\gm_2=0$ was obtained by Bai-Liu \cite{BL2019} when $\be>1$ and  $\mu\in \R$, posed on non-trapping asymptotically Euclidean manifolds \eqref{dl1}-\eqref{eq-g2-ae2}.
Notice that, the global existence part does not have the nonnegative assumption on $b(t)$, so it is natural to infer that the nonnegative assumption might not be necessary for the blow up results as well.
 
 In our recent work \cite{LW2019}, where we consider the finite time blows up of solutions to \eqref{nlw2} with $c_1=0<c_2$, we have succeeded in proving blow up results without the sign conditions on $b$, by applying a variable transform. Let
\beeq
 \label{bianhuan}
 m(t)=e^{\int^{t}_{0}b(\tau) d\tau}, \quad s=\int^{t}_{0}\frac{1}{m(\tau)}d\tau=h(t), t=\eta(s)\ ,
 \eneq
then \eqref{nlw2} becomes
$$\pa^{2}_{s}u-\tilde{m}^{2}(s)\Delta_{\gm}u=c_1 \tilde{m}^{2-p}(s)|u_{s}|^{p}
+c_2 \tilde{m}^{2}(s)|u|^{q}
, \quad \tilde{m}(s)=m(\eta(s))\ .
$$
Traditionally, we still use $t$ to denote time, and we are reduced to consider the following equation
\beeq
\label{bianhuan1}
\pa^{2}_{t}u-\tilde{m}^{2}(t)\Delta_{\gm}u=
c_1 \tilde{m}^{2-p}(t)|u_{t}|^{p}
+c_2 \tilde{m}^{2}(t)|u|^{q}, \quad \tilde{m}(t)=m(\eta(t))\ .
\eneq

Based on this transformation, as is typical for the test function method of proving blow up, 
one of the  key points is to find appropriate test function adapted for the problem.
For that purpose, 
we  state the main assumption that we shall make.
\\
{\bf Hypothesis:} 
There exist  $\la, c>0$ and corresponding nonnegative nontrivial entire solution $\phi$ to $\Delta_{\gm}\phi=\la^{2}\phi$ such that
\beeq\tag{H}
  \label{H1}
 0\leq\phi(x)\le c\< \la r\>^{-\frac{n-1}{2}}e^{\la \int^{r}_{0}K(\tau)d\tau}, \ \phi\neq 0\ .
 \eneq

Let us review some cases where the assumption \eqref{H1} is valid. It is valid on Euclidean space, $\gm=\gm_0$, where
$\phi$ with $\la>0$ could be given by spherical average of $e^{\la x\cdot\om}$,
$$\phi=\int_{\mathcal{S}^{n-1}}e^{\la x\cdot\omega}d\omega\ ,$$
see Yordanov-Zhang \cite{YorZh06}.
When $\gm_3$ is  an exponential perturbation, that is, there exists $\al>0$ so that
\beeq
\label{dlfjia}
|\nabla g_{3,jk}|+|g_{3,jk}|\les  e^{-\al\int^{r}_{0}K(\tau)d\tau}\ ,
\eneq
the assumption \eqref{H1} is recently verified for
 $\gm=\gm_1+\gm_3$ by the authors \cite{LW2019} with uniform positive lower bound of $\phi$, while the 
case
 $\gm=\gm_0+\gm_3$  
was obtained by
Wakasa and Yordanov in \cite{WaYo18-1pub}.

 The main result of this paper then states that there are no global solutions to
 \eqref{nlw} or \eqref{nlw2},
 for arbitrary small $\ep>0$, provided that
$p\le p_G$, or $q\leq p_S$, or 
\beeq\label{mixed}
(q-1)((n-1)p-2)<4,
 p>1,  q >1\ .\eneq
More precisely, we have
\begin{thm}
\label{main-a-g}
Let $n\geq 2$, $p, q>1$  and $b(t)\in L^{1}(\R_+)$. Assuming that \eqref{H1} is valid. Consider \eqref{nlw2} with $c_1, c_{2}\ge 0$ and nontrivial initial data \eqref{hs2}, posed on asymptotically Euclidean manifolds \eqref{dl1}-\eqref{eq-g2-ae2}. Suppose it has a weak solution $u\in C^{2}([0, T_\ep);\mathcal{D}'(\R^{n}))$ with $u_{t}\in L_{loc}^{p}([0, T_{\ep})\times \R^{n})$, $|u|^{q}\in C([0, T_\ep);\mathcal{D}'(\R^{n}))$ and
\beeq
\label{fsp}
\supp u\subset \{(t, x); \int^{r}_{0}K(\tau)d\tau\leq t+R_{1}\},
\eneq
for some $R_{1}\ge \int_{0}^{R_{0}}K(\tau)d\tau$. Then 
there exist constants $\ep_0>0$ and $C_{0}>0$, such that for any $\ep\in (0, \ep_{0})$, we have the following results on the upper bound of $T_\ep$:
\begin{enumerate}
  \item \eqref{main-es}, if $c_{1}>0$,  $1<p\leq p_{G}$,
  \item \eqref{mixed-lifespan1}, if $c_1c_2>0$ and \eqref{mixed},
  \item \eqref{eq-life},
 if $c_{2}>0$,  $1<q< p_{S}$\footnote{When $q=p_S$, the situation is more delicate. It can be shown that
 we still have \eqref{eq-life}, provided that
there exist $\la_0, c_1>0$ such that we have a class of solutions to $\Delta_{\gm}\phi_\la=\la^{2}\phi_\la$ satisfying
the uniform bound $$
c_{1} 
<\phi_\la(x) < c_1^{-1}\< \la r\>^{-\frac{n-1}{2}}e^{\la \int^{r}_{0}K(\tau)d\tau} \ ,
$$
 for any
$0<\la\le \la_0$.
 See \cite{LW2019} for a proof. 
},
\item \eqref{eq-life-special} if $c_{2}>0$, $n=2$,  $1<q<2$ and  $u_1\neq 0$.
\end{enumerate}
\end{thm}
\begin{rem}
As we have reviewed, in the previous works, the typical assumptions on the data are
\eqref{hs2} and $u_1\neq 0$ for
 \eqref{main-es} and
 \eqref{mixed-lifespan1}.
Here, we observe that we could actually relax the assumption to any nontrivial initial data
with \eqref{hs2}. It has the obvious benefit that we could then compare upper bounds between $S_\ep$, $G_\ep$ and $Z_\ep$.
\end{rem}
As is clear, the upper bounds \eqref{eq-life} and \eqref{eq-life-special} could be easily adapted from the corresponding proof for the results for the case $c_1=0$, and we refer
 \cite{LW2019} for the proof.

For the strategy of proof  of \eqref{main-es} and
 \eqref{mixed-lifespan1}, we basically use the test function method, with the help of \eqref{H1} and the transformation \eqref{bianhuan}. 
We should note that, it seems that the test function $\tilde{\psi}(t,x)=e^{-\la\eta(t)}\phi$ used in \cite{LW2019} does not work equally well for the nonlinear term $|u_{t}|^{p}$, unless one assume $b(t)\geq 0$. To avoid this difficulty, we work directly on the precise solution of linear ODE
 $$\psi_{tt}-\tilde{m}^{2}\psi=0\ .$$
In this case, we do not have the explicit formula of the solution. Nevertheless, we can obtain the desired asymptotic behavior of the solution by applying the Levinson theorem and ODE arguments, which is still sufficient for our proof.

As we see, our  result makes sense under the assumption that there exists a distribution solution that satisfies finite speed of propagation. Once we have the local well-posedness, we can remove this technical assumption. 
Before concluding the introduction,
 we discuss some situations where we do have local well-posedness.

At first, 
 the local energy and KSS estimates are available for nontrapping asymptotically Euclidean manifolds,
see Bony-H{\"a}fner \cite{BoHa} and Sogge-Wang \cite{sw2010}.
With help of these estimates, we could prove not only local well-posedness, but also almost global existence, 
for the  problem when $n=3$, $p=p_G=2$ with $c_2=0$.
It has been obtained in \cite{sw2010}, \cite{W15} for the case $b=0$.
As the argument in \cite{sw2010}, \cite{W15} could be easily adapted to the current setting, with additional absorption of the term $b(t)\in L^{1}(\R_+)$ by Gronwall's inequality, similar to that in \cite{LWDCDS}, we omit the proof.

 \begin{thm}
\label{main-2} 
  Let $n=3$ and $p=2$. Consider \eqref{nlw2} with $c_2=0$, $b(t)\in L^{1}(\R_+)$, posed on non-trapping asymptotically Euclidean manifolds \eqref{dl1}-\eqref{eq-g2-ae2}. Assuming that \eqref{H1} is valid. Then for any 
nontrivial initial data with \eqref{hs2},  there is $\ep_0>0$ so that for any $\ep\in (0, \ep_0)$, we have a unique local weak solution $(u, u_{t})\in C([0,T];H^{3}\times H^{2})$. Moreover, we have the following estimate for the lifespan $T_\ep$
  $$\exp(c_{0}/\ep)\leq T_{\ep}\leq \exp(C_{0}/\ep),$$ 
for some $0<c_{0}<C_{0}$.
   \end{thm}

Recall that
as a fundamental tool to prove local well-posedness,
 Strichartz estimates for wave operators with variable $C^{1,1}$ coefficients have been well-understood, see
 Smith \cite{Smith1998}, Tataru \cite{Ta02} and references therein.
When $2\leq n\leq 4$, it turns out that we could use
Strichartz estimates to prove
 local well-posedness of \eqref{nlw2} on general Riemannian manifolds, 
 which apply for
 asymptotically Euclidean manifolds \eqref{dl1}-\eqref{eq-g2-ae2}.
 \begin{thm}
 \label{main-3}
  Let $2\leq n\leq 4$, $p>\max(1, (n-1)/2)$,
 $q>\max(1, (n-1)/2, n/2-1/(p-1))$. The problem \eqref{nlw2} with $b(t)\in L^{1}(\R_{+})$
 and nontrivial data, posed on $C^{3}$
Riemannian $(\R^n,\gm)$  with
$\nabla^\be g_{jk} \in L^{\infty}$,  $|\be|\le 3$,
 is locally well posed in $H^{s+1}\times H^{s}$ for small enough $\ep>0$,
 where  
\beeq\label{eq-s}
  s\in \Big(\max(\frac{n-1}{2},
  \frac{n+1}{4}, \frac{n}2
  -\frac{1}{p-1}), \min(2, p, q)\Big)\ .
\eneq
  In addition, if 
the initial data satisfies \eqref{hs2},
  $\gm=\gm_1+\gm_3$ with
\eqref{dl3}-\eqref{dl2} and 
\eqref{dlfjia},
we have upper bounds of the lifespan as in
Theorem \ref{main-a-g}.
   \end{thm}
   See Figure \ref{tu1} for a comparison between region for local well posedness and blow up. 

 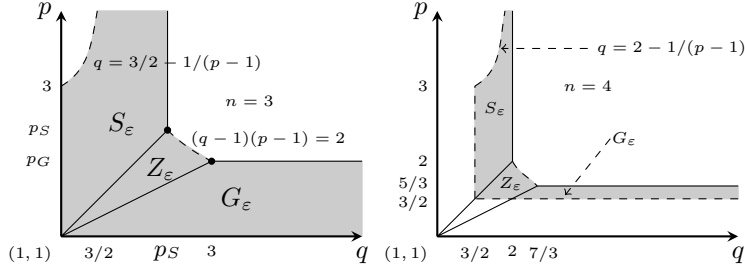
\begin{figure}
\centering
\begin{tikzpicture} 

\filldraw[black!20!white] (1,1)--(1,3)
to[out=30,in=260] (1.48,4)--(2.414,4)
--(2.414,2.414)
to[out=297,in=134](2.6,2.25)-- (3, 2)--(5,2)--(5,1);--(1,1);

\draw[dashed](1,3) to[out=30,in=260] (1.48,4);
\node[right] at (1.3,3.3) {\tiny $q=3/2-1/(p-1)$};
\draw (2.414,2.414)--(2.414,4) (3,2)--(5,2);
\draw[ dashed, domain=2.414:3] plot (\x, {{1+2/((\x)-1)}});
\node[right] at (2.6,2.3) {\tiny $(q-1)(p-1)=2$};
\node[right] at (1.5,2.5) {$S_{\ep}$};
\node[right] at (3,1.5) {$G_{\ep}$};
\draw[fill=black,line width=1pt] (2.414,2.414) circle[radius=0.3mm];
\draw[fill=black,line width=1pt] (3,2) circle[radius=0.3mm];
\node[below left] at (1,1) {\tiny $(1,1)$};
\node[left] at (1,2.414) {\tiny $p_S$};
\node[right] at (2,1.9) {$Z_{\ep}$};
\node[left] at (1,2) {\tiny $p_G$};
\node[below] at (2.414,1) {\small $p_S$};
\node[below] at (1.5,1) {\tiny $3/2$};
\node[left] at (1,3) {\tiny $3$};
\node[below] at (3,1) {\tiny $3$};
	\draw (1,1)--(2.414,2.414) (1,1)--(3,2);
	\draw[thick,-stealth] (1,1)--(1,4) node[left]{$p$};
	\draw[thick,-stealth] (1,1)--(5,1) node[below]{$q$};
\node[below] at (3.5,3) {\tiny $n=3$};

\node[left] at (6,3) {\tiny $3$};
\filldraw[black!20!white] (6.5,1.5)--(6.5,3)
to[out=30,in=260] (6.9,4)--(7,4)
--(7,2)
to[out=297,in=134] (7.33, 1.67)--(10,1.67)--(10,1.5)--(6.5,1.5);
\draw[dashed](6.5,3) to[out=30,in=260] (6.9,4) (7,2)
to[out=297,in=134] (7.33, 1.67) (6.5,1.5)--(6.5,3) (10,1.5)--(6.5,1.5);
	\draw (7,2)--(7,4) 	(7.33, 1.67)--(10,1.67) (6,1)--(7,2) (6,1)--(7.33,1.67);
	\draw[thick,-stealth] (6,1)--(6,4) node[left]{$p$};
	\draw[thick,-stealth] (6,1)--(10,1) node[below]{$q$};
\node[below left] at (6,1) {\tiny $(1,1)$};
\node[left] at (6,2) {\tiny $2$};
\node[right] at (6.7,1.7) {\tiny $Z_{\ep}$};
\node[right] at (8.2,2.3) {\tiny $G_{\ep}$};
	\draw[->, dashed] (8.3,2.3)--(7.7, 1.55);
\node[right] at (8,3.5) {\tiny $q=2-1/(p-1)$};
	\draw[->, dashed] (8,3.5)--(6.85, 3.5);

\node[right] at (6.5,2.7) {\tiny $S_{\ep}$};
\node[left] at (6, 1.7) {\tiny $5/3$};
\node[left] at (6, 1.45) {\tiny $3/2$};
\node[below] at (7.414,1) {\tiny $7/3$};
\node[below] at (7,1) {\tiny  $2$};
\node[below] at (6.5,1) {\tiny $3/2$};
\node[below] at (8,3.2) {\tiny $n=4$};

\end{tikzpicture}
 \caption{Local well posedness vs blow up, when $n=3, 4$}
\label{tu1}
\end{figure}

For high dimensional case, as $\min (p,q)$ is close to $1$, it seems difficult to employ the approach of using Strichartz estimates 
 obtain desired well-posed results. Instead, in the case of 
spherical symmetric, small metric perturbation, 
that is, $\gm_{2}=0$ and there exists
$\theta>0$,  such that
\beeq
\label{1029}
\sum_{j\in \Z} 2^{jk}\|\pa_r^k(K-1)\|_{L^{\infty}_{r}(r\sim 2^j)}\le \theta , k=0,1, 2\ ,
\eneq
it turns out that the homogeneous local energy and KSS-type estimates are available 
(see Metcalfe-Sogge \cite{MS2006} and Hidano-Wang-Yokoyama \cite[Lemma 2.3]{hyw1}).
Thus, in spirit of
\cite{hyw} and \cite{HJWLax},
we could prove well-posed results for the problem with $c_{2}=0$ in the radial case, with sharp lower bound of the lifespan.

 \begin{thm}
 \label{main-4}
 Let $n\geq 3$, $1<p\leq p_{G}$ and $\gm=\gm_{1}$ with \eqref{1029}. 
 Consider \eqref{nlw2} with $b(t)\in L^{1}(\R_+)$ and $c_2=0$. 
 Then there exists $\theta_0>0$ such that the following statement holds for $\theta\le \theta_0$:
 for any radial, nontrivial initial data 
satisfying \eqref{hs2}, there is $\ep_1>0$ so that for any $\ep\in (0, \ep_1)$, the problem
\eqref{nlw2}  admits
a unique local weak solution $(u, u_{t})\in C([0,T];H_{rad}^{2}\times H_{rad}^{1})$.  Moreover, let $T_\ep$ be the lifespan of the local solution, we have 
\beeq
\begin{cases}
c\ep^{-\frac{2(p-1)}{2-(n-1)(p-1)}}\leq T_{\ep}\leq C\ep^{-\frac{2(p-1)}{2-(n-1)(p-1)}}, \ 1<p<p_{G}\\
\exp{(c\ep^{-p(p-1)})}
\leq T_{\ep}\leq \exp{(C\ep^{-p(p-1)})},\ p=p_G\ 
\end{cases}
\eneq  
for some $0<c<C$.
Here, we use $H^{s}_{rad}$ to denote the space of spherically symmetric functions in the usual Sobolev space $H^{s}$.
\end{thm}    
 Finally, we give the local well posedness for $c_1c_2\neq 0$ when $n\geq 5$ by exploiting local energy estimates. 
\begin{thm}
 \label{main-5}
 Let $n\geq 5$, $1<p<1+\frac{2}{n-2}$, $1<q<1+\frac{4}{n-4}$. Assuming $\gm=\gm_{1}$ with \eqref{1029}. 
 Consider \eqref{nlw2} with $b(t)\in L^{1}(\R_+)$. 
 Then there exists $\theta_1>0$ such that the following statement holds for $\theta\le \theta_1$:
 for any nontrivial radial initial data 
satisfying \eqref{hs2}, there is $\ep_1>0$ so that for any $\ep\in (0, \ep_1)$, the problem
\eqref{nlw2}  admits
a unique local weak solution $(u, u_{t})\in C([0,T];H_{rad}^{2}\times H_{rad}^{1})$.  Moreover, we have upper bounds of the lifespan as in
Theorem \ref{main-a-g}.
\end{thm}  
See Figure \ref{tu2} for a comparison between region for local well posedness and blow up. 

 \begin{figure}
\centering
\begin{tikzpicture} 
\filldraw[black!20!white] (1,1)--(1,2.5)--(3.027, 2.5)--(3.5,2) to[out=297,in=134] (3.6, 2)--(4,2)--(4,1);--(1,1);
\draw[thick] (2.8,2.8)--(2.8,3.2) (3.5,2)--(5,2);
\draw[thick, dashed, domain=2.8:3.5] plot (\x, {{11.2/(\x)-1.2}});
\node[right] at (3.2,2.7) {\small $(q-1)((n-1)p-2)=4$};
\draw[fill=black,line width=1pt] (2.8,2.8) circle[radius=0.3mm];
\draw[fill=black,line width=1pt] (3.5,2) circle[radius=0.3mm];
\node[below left] at (1,1) {\small $(1,1)$};
\node[left] at (1,2.8) {$p_S$};
\node[left] at (1,2.4) {\small $1+\frac{2}{n-2}$};
\node[left] at (1,2) {$p_G$};
\node[below] at (2,1) {$p_G$};
\node[below] at (2.8,1) {$p_S$};
\node[below] at (4.1,1) {\small $\frac{n}{n-4}$};
\node[below] at (3.4,1) {\small $\frac{n+3}{n-1}$};
	\draw[thick, dashed] (1,2.5)--(4,2.5) (4,1)--(4,2.5);
	\draw[thick,-stealth] (1,1)--(1,3.2) node[left]{$p$};
	\draw[thick,-stealth] (1,1)--(5,1) node[below]{$q$};

\end{tikzpicture}
 \caption{Local well posedness vs blow up, when $n\geq 5$}
\label{tu2}
\end{figure}
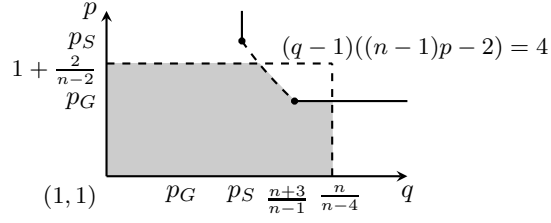

This paper is organized as follows. In the next section, we give the existence part of
Theorems \ref{main-2}-\ref{main-5},
i.e., local existence and uniqueness of weak solutions, which ensure the property of finite speed of propagation \eqref{fsp}, instead of the assumption. In Section 3, we give the proof of Theorems \ref{main-a-g}-\ref{main-5}
 with $c_1>0$, that is, \eqref{main-es} for $1<p\leq p_G$. In the last Section 4, we present the proof of the blow up part, \eqref{mixed-lifespan1}, for the case of mixed nonlinearities, with $c_1c_2\neq 0$.

We close this section by listing some notations. For a norm $X$ and a nonnegative integer $m$, we will use the shorthand notation
\[|\pa^{\le m} u| = \sum_{|\mu|\le m} |\pa^\mu u|,\quad \|\pa^{\le m} u\|_X
= \sum_{|\mu|\le m} \|\pa^\mu u\|_X.\]
For $1\leq r\leq \infty$ and $s\in \R$, we denote the norm  
$$\|u\|_{\dot{l}_{r}^s(X)}=\|(\phi_j(x)u(t,x))\|_{\dot{l}^{s}_r(X)} = \|(2^{js}\|\phi_j(x)u(t,x)\|_{X})\|_{\ell^{r}(j\in \mathbb{Z})},$$
for a partition of unity subordinate to the dyadic spatial annuli, $\sum_{j\in \mathbb{Z}}\phi_j^2(x)=1$, for any $x\neq0$.

\section{Local well-posedness}\label{section2}
In this section, we present the proof of local well-posedness for equation \eqref{nlw2}. In the following, we recall the trace estimates we will use later.
\begin{lem}[Trace estimates]
\label{trace} Let $n\geq 2$, $1/2<s<n/2$. Then 
$$\|r^{n/2-s}u\|_{L_r^{\infty}L_{\omega}^{2}}\les \|u\|_{\dot{H}^s}, \|r^{(n-1)/2}u\|_{L_r^{\infty}L_{\omega}^{2}}\les \|u\|_{H^s},$$
in particular, for radial functions,
$$\|r^{n/2-s}\langle r \rangle^{s-1/2}u\|_{L_x^\infty}\les \|u\|_{H_{rad}^s}\ .$$
\end{lem}
See, e.g., Fang-Wang \cite{FW2011} or Wang-Yu \cite{WY2012}.

\subsection{Strichartz estimates}
\begin{lem}[Local homogeneous Strichartz estimates]
\label{pro-11-3}
Let  $n\ge 2$, $b\in L^1$,
$(\R^n,\gm)$ be $C^2$ Riemannian metric with
$
\nabla^\be g_{jk} \in L^{\infty}$,  $|\be|\le 2$. Then 
for $(s, q, r)$ admissible with $r<\infty$, that is 
$$\frac{1}{q}\leq \min\left(\frac{1}{2}, \frac{n-1}{2}
\left(\frac{1}{2}-\frac{1}{r}\right)\right), s=n\left(\frac{1}{2}-\frac{1}{r}\right)-\frac{1}{q}\ ,$$
we have the local homogeneous Strichartz estimates 
\beeq
\label{hsz}
\|\pa D^{-s}u\|_{L^{q}_{t\in[0, 1]}L_{x}^{r}}
+\|\pa u\|_{L^\infty_{t\in [0,1]} L^{2}}
\les \|\pa u(0)\|_{L^{2}}\ ,
\eneq
for solutions to 
$(\pt^2-\Delta_\gm+b(t)\pt)u=0$.
\end{lem}
See Smith \cite{Smith1998} ($n=2, 3$) and Tataru \cite[Theorem 1.1]{Ta02} for the case $b(t)=0$. The result when $b(t)\in L^{1}$ follows directly if we combine it with the Duhamel principle and the Gronwall inequality. 
\begin{coro}
Let $2\le n\le 4$,
$(\R^n,\gm)$ be $C^3$ Riemannian metric with
$
\nabla^\be g_{jk} \in L^{\infty}$,  $|\be|\le 3$,
 $\max{(2, 4/(n-1))}<q<\infty$ and $\tau\in(\frac{n}{2}-\frac{1}{q}, 2]$.
Then
the solution of equation $(\pt^2-\Delta_\gm+b(t)\pt)u=F$ satisfies
\beeq
\label{NoHStri}
\|\pa u\|_{L^q_{t\in [0,1]} L_x^{\infty}}+
\|\pa u\|_{L^{\infty}_{t\in[0, 1]}H^{\tau}}\les\|\pa u(0)\|_{H^{\tau}}+\|F\|_{L^{1}H^{\tau}}\ .
\eneq

\end{coro}
\begin{proof}
We need only to give the proof for $F=0$, by Duhamel's principle.
By applying $\nabla^\al$, $|\al|\le 2$, to $\pa^{2}_{t}u-\Delta_{\gm}u+b(t)u_{t}=0$,
we get
$$(\pa^{2}_{t} -\Delta_{\gm}+b(t) \pt)\nabla^{\al} u=
[-\Delta_{\gm},\nabla^{\al}] u=\mathcal{O}(
\sum_{1\leq |\be|\leq 1+|\al|}|\nabla^\be u|
)\ .$$
Then by \eqref{hsz} and Gronwall's inequality, we have 
\beeq
\label{eq-11-4}
\|\pa D^{2-s}u\|_{L^{q}_{t\in[0, 1]}L_{x}^{r}}
+\|\pa u\|_{L^\infty_{t\in [0,1]} H^{2}}
\les \|\pa u(0)\|_{H^{2}}\ .
\eneq
By interpolation, we know that 
$$
\|\pa  D^{{\tau}-s} u\|_{L^{q}_{t\in[0, 1]}L_{x}^{r}}
+\|\pa u\|_{L^\infty_{t\in [0,1]} H^{{\tau}}}
\les \|\pa u(0)\|_{H^{{\tau}}}\ ,
$$
for any $\tau\in [0, 2]$.
If $\tau-s>\frac{n}{r}$, that is, $\tau>\frac{n}{2}+s=\frac{n}{2}-\frac{1}{q}$,
by the Sobolev inequality, we have
$$\|\pa u\|_{L^q_{t\in [0,1]} L_x^{\infty}}\leq \|\pa D^{\tau-s}u\|_{L^{q}_{t\in[0, 1]}L_{x}^{r}}\les \|\pa u(0)\|_{H^{{\tau}}}\ ,$$
which completes the proof.
\end{proof}
 \begin{proposition}\label{prop-Stri}
 Let $2\leq n\leq 4$, $p>\max(1, (n-1)/2)$,
 $q>\max(1, (n-1)/2, n/2-1/(p-1))$. The problem \eqref{nlw2} with $b(t)\in L^{1}(\R_{+})$, posed on 
Riemannian $(\R^n,\gm)$  with
$\nabla^\be g_{jk} \in L^{\infty}$,  $|\be|\le 3$,
 is locally well posed in $H^{s+1}\times H^{s}$,
 where  
\beeq\label{eq-s}
  s\in \Big(\max(\frac{n-1}{2},
  \frac{n+1}{4}, \frac{n}2
  -\frac{1}{p-1}), \min(2, p, q)\Big)\ ,
\eneq
for small enough $\ep>0$.
 \end{proposition}
\begin{proof} 
For fixed $s$ satisfying
\eqref{eq-s}, there is
$r>\max (2, 4/(n-1), p-1)$ such that
$s\in (n/2-1/r, 2)$ so that we could apply \eqref{NoHStri}.
Let
$$\|u\|_{X^s}=
\| u\|_{L^\infty_{t\in [0,1]} H^{s+1}}+
\|u_{t}\|_{L^\infty_{t\in [0,1]} H^{s}}+
\|\pa u\|_{L_t^{r}( [0, 1])L_x^{\infty}}\ .$$
For any $u\in X^{s}$, we define $\Pi u$ to be the solution to the linear equation
\beeq
\left\{\begin{array}{l}(\pa^{2}_{t} - \Delta_{\gm} +b(t)\pt)\Pi u=F(u):=c_1|u_t|^{p}+c_{2}|u|^{q},\quad
\\
\Pi u(0,x) = \ep u_0,
 \pt  \Pi u(0,x) = \ep u_1.\end{array}
 \right.
 \label{11-3-1}
\eneq
In view of the Strichartz estimate \eqref{NoHStri}, 
we have
$$\|\Pi u\|_{X^{s}}\les
\ep(\|u_{0}\|_{H^{s+1}}+\|u_{1}\|_{H^{s}})+
\|F(u)\|_{L^{1}H^{s}}\ ,
$$
$$\|\Pi u-\Pi w\|_{L^{\infty}L^{2}}\les
\|F(u)-F(w)\|_{L^{1}L^{2}}
\ .
$$
Then,  when $\ep$ is small enough, by a standard iteration argument,
 the proof of local well posedness for $t\in [0, 1]$
  is reduced to the proof of the following two
 nonlinear estimates
\beeq\label{eq-Stri1}
\|F(u)\|_{L^{1}H^{s}}
\les \|u\|_{X^{s}}^{p}+\|u\|_{X^{s}}^{q}\ ,
\eneq
\beeq\label{eq-Stri2}
\|F(u)-F(w)\|_{L^{1}L^{2}}
\les 
\|(u,w)\|_{X^{s}}^{p-1}
\| \pt u-\pt w\|_{L^{\infty}L^{2}}
+\|(u,w)\|_{X^{s}}^{q-1}
\| u- w\|_{L^{\infty}L^{2}}
\ .
\eneq

The second inequality
\eqref{eq-Stri2} follows directly from the Sobolev embedding and H\"older's inequality,
as $s+1>n/2$ and  $r>p-1$.
Concerning \eqref{eq-Stri1}, since $p, q>s$ and $s+1>n/2$,
by the Moser type inequality, see, e.g., Hidano-Wang \cite[Lemma 2.3]{HW2019}, we have 
$$\|F(u)\|_{H^{s}}\les \|u_t\|^{p-1}_{L^{\infty}}\|u_t\|_{H^{s}}
+\|u\|^{q-1}_{L^{\infty}}\|u\|_{H^{s}}
\les
 \|u_t\|^{p-1}_{L^{\infty}}\|u_t\|_{H^{s}}
+\|u\|^{q-1}_{H^{s+1}}\|u\|_{H^{s}}
\ ,$$
and so, as $r>p-1$,
$$\|F(u)\|_{L^{1}H^{s}}
\les
 \|u_t\|^{p-1}_{L^{r}L^{\infty}}\|u_t\|_{L^{\infty}H^{s}}
+\|u\|^{q-1}_{L^{\infty}H^{s+1}}\|u\|_{L^{\infty}H^{s}}
\les \|u\|_{X^{s}}^{p}+\|u\|_{X^{s}}^{q}
\ ,$$
which is \eqref{eq-Stri1}. This completes the proof.
\end{proof}

\subsection{The Glassey conjecture for small, radial asymptotically Euclidean manifolds} 

As we mentioned before,
when $\gm=\gm_{0}$ and $n\geq 2$, 
the local well-posedness and long time existence for \eqref{nlw2} with $b=c_2=0$ has been studied in
\cite{hyw} and \cite{HJWLax}, 
by exploiting Morawetz type local energy estimates and its variants
(see, e.g., \cite{KSS2002}, 
\cite{MS2006}, \cite{hyw1}), together with the trace estimates.
The similar approach could apply in the setting of small, radial asymptotically Euclidean manifolds, at least when $n\ge 3$.

In practice, the following estimates are useful in proving well-posedness for
\eqref{nlw2},
\beeq
\label{11-2-6}
T^{\frac{\mu-1}{2}}\|r^{-\frac{\mu}{2}}\tilde\pa u\|_{L_{T,x}^{2}}+ (\ln(T+2))^{-\frac{1}{2}}\|r^{-\frac{\mu}{2}}\langle r \rangle^{-\frac{1-\mu}{2}}\tilde\pa u\|_{L_{T,x}^{2}}\les \|u\|_{X_T}\ , \mu\in (0, 1)\ ,
\eneq
where, we denote $L_{T}^{q}$ as $L_{t}^{q}([0, T])$,
$\tilde \pa u=(\pa u, u/r)$
 and
$$\|u\|_{X_T}:=\|\tilde \pa u\|_{\dot{\ell}^{-\frac{1}{2}}_{\infty}L_{t,x}^{2}\cap L_t^{\infty}L_x^2([0,T]\times\R^n)},
\|F\|_{N_T}:=\|F\|_{\dot{\ell}^{\frac{1}{2}}_{1}L_{t,x}^{2}+ L_t^{1}L_x^2([0,T]\times\R^n)}
\ .$$
See, e.g., \cite{hyw1}, for a proof.
Concerning $X_T$ and $N_T$, we have the following version of the local energy estimates.
\begin{lem}\label{lem-LE}
Let $n\ge 3$, $b\in L^1$,
$(\R^n,\gm)$ with $\gm=\gm_{1}$ satisfying \eqref{1029} and $\theta$ small enough.
Then there exists a constant $C>0$ such that
\beeq
\label{LoES}
\|u\|_{X_T}\le C \|\pa u(0)\|_{L^2}+C\| F\|_{N_T}\ ,
\eneq
\beeq
\label{LoES2}
\|\nabla^{\leq 1} u\|_{X_T}\le C \|\pa u(0)\|_{H^1}+C\|\nabla^{\leq 1} F\|_{N_T}\ .
\eneq
for any solutions to $(\pt^2-\Delta_\gm+b(t)\pt)u=F$ in $[0,T]\times\R^n$.
\end{lem}
\begin{proof}
Recall that, when $\gm=\gm_{1}$ satisfies \eqref{1029} for $k=0,1$ with
 $\theta$ small enough, 
the inequality \eqref{LoES}
has essentially been proved in Metcalfe-Sogge \cite{MS2006} and Hidano-Wang-Yokoyama \cite[Lemma 2.3]{hyw1} for $b=0$, while the general case with $b\in L^1$ follows from the Gronwall inequality. 

Since $$(\pa^{2}_{t}-\Delta_{\gm}+b(t)\pt) \nabla u= \nabla F +[\nabla, \Delta_{\gm}]u=\nabla F+\mathcal{O}( |(\nabla^{2}g) (\nabla u)|+|(\nabla g) (\nabla^2 u)|)\ ,$$
 by \eqref{LoES}, we have 
\begin{align*}
\| \nabla u\|_{X_T}\les &\|\pa u(0)\|_{\dot{H}^1}+\|\nabla F\|_{N_T}+\|\nabla^{2}g \nabla u\|_{\dot{\ell}^{\frac{1}{2}}_{1}L_{t,x}^{2}}+\|\nabla g \nabla^2 u\|_{\dot{\ell}^{\frac{1}{2}}_{1}L_{t,x}^{2}}\\
\les &\|\pa u(0)\|_{\dot{H}^1}+\| \nabla F\|_{N_T}+
\| \tilde\pa \nabla g\|_{\dot{\ell}^{2}_{1}L_{x}^{\infty}} \|\nabla u\|_{X_T}.
\end{align*}
As $\| \tilde\pa \nabla g\|_{\dot{\ell}^{2}_{1}L_{x}^{\infty}}\ll 1$ by assumption \eqref{1029}, we obtain 
 \eqref{LoES2}.
\end{proof}

\begin{proposition}
\label{lem-2.4}
Let $n\geq 3$ and $1<p\leq p_{G}$. Consider \eqref{nlw2} on asymtotically Euclidean manifolds $\gm$ with $\gm=\gm_{1}$ satisfying \eqref{1029} and $\theta$ small enough. Then \eqref{nlw2} with $c_2=0$ is locally well-posed in $H_{rad}^2\times H_{rad}^1$ for small enough $\ep$. Moreover, 
there exists $c>0$ such that,
 for any initial data $(u_0, u_1)\in H_{rad}^2\times H_{rad}^{1}$ with 
$$\|u_0\|_{H^2}+\|u_1\|_{H^1}\leq 1,$$ there exists a
 unique solution $u\in C([0, T_{\ep}]; H_{rad}^{2})\cap C^{1}([0, T_{\ep}]; H_{rad}^{1})$ with
 \beeq
T_{\ep}\ge
\begin{cases}
c \ep^{-\frac{2(p-1)}{2-(p-1)(n-1)}},\ 1<p<p_{G},\\
\exp(c \ep^{-(p-1)}),\ p=p_G\ .
\end{cases}
\eneq
 \end{proposition}
\begin{proof}
We give the proof using the similar approach in
Proposition \ref{prop-Stri}.
For the operator $\Pi$ defined in
\eqref{11-3-1}, in view of the linear estimates
in Lemma \ref{lem-LE},
we have
$$\|\nabla^{\le 1} \Pi u\|_{X_{T}}\le
C\ep+C
\|\nabla^{\le 1}F(u)\|_{N_T}\ ,
$$
$$\|\Pi u-\Pi w\|_{X_T}\le C
\|F(u)-F(w)\|_{N_T}
\ .
$$
Then
we need only to prove the following nonlinear estimates for radial functions $u, w$,
\beeq\label{eq-LE1}
\|\nabla^{\le 1}|u_t|^p\|_{N_T}
\les A(T) \|\nabla^{\le 1}u\|_{X_T}^{p}\ ,
\eneq
\beeq\label{eq-LE2}
\||u_t|^p-|w_t|^p\|_{N_T}
\les 
A(T) \|\nabla^{\le 1} (u,w)\|_{X_T}^{p-1}
\| u-w\|_{X_T}
\ ,
\eneq
where
$$A(T)=\left\{
\begin{array}{ ll}
T^{1-\mu}      &  1<p<p_G, \mu=\frac{(p-1)(n-1)}{2}\in(0, 1) , \\
\ln(2+T)      &   p=p_G, \mu=\frac{n-2}{n-1}\ .
\end{array}
\right.$$
Notice that the lower bound of  $T_\ep$ is obtained from the requirement
$A(T_\ep) \ep^{p-1}\ll 1$, in the process of closing the iteration.

These nonlinear estimates are well-known, see, e.g., \cite{hyw} and \cite{HJWLax}.
For reader's convenience, we give a proof of
\eqref{eq-LE1} here.
By Lemma \ref{trace}, we have
$$\|r^{(n-1)/2} u_t\|_{L^\infty}
\les \|r^{(n-2)/2}\<r\>^{1/2} u_t\|_{L^\infty}
\les \|u_t\|_{H^1}\ .$$
When $p\in (1, p_G)$, we have
$\mu/(p-1)=(n-1)/2$ and so by \eqref{11-2-6}, we see that
$\|\nabla^{\le 1}|u_t|^p\|_{N_T}$ is controlled by
$$
 T^{(1-\mu)/2}
\|r^{\mu/2}\nabla^{\le 1}|u_t|^p\|_{L^2_{T,x}}
\les T^{1-\mu} \|r^{\mu/(p-1)} u_t\|_{L^\infty}^{p-1}
\|\nabla^{\le 1}u\|_{X_T}
\les A(T) \|\nabla^{\le 1}u\|_{X_T}^{p}\ .
$$
When $p=p_G$, we have
$\mu/(p-1)=(n-2)/2$ and then
\begin{eqnarray*}
 \|\nabla^{\le 1}|u_t|^p\|_{N_T}& \les & (\ln(2+T))^{1/2}
\|r^{\mu/2}\<r\>^{(1-\mu)/2}\nabla^{\le 1}|u_t|^p\|_{L^2_{T,x}}
 \\
 & \les & (\ln(2+T))
\|r^{\mu}\<r\>^{1-\mu}|u_t|^{p-1}\|_{L^\infty_{T,x}}
\|\nabla^{\le 1}u\|_{X_T}
\\
&\les &A(T) \|\nabla^{\le 1}u\|_{X_T}^{p}\ .
\end{eqnarray*}
This completes the proof.
\end{proof}

\subsection{Local well-posedness for $c_1c_2\neq 0$ in high dimension}When $n\geq 5$, it seems that the approach of using Strichartz estimate does not apply for small $p, q$. However, the local energy estimates in Lemma \ref{lem-LE} works well for some powers. In summary, we can show \eqref{nlw2} is locally well posed, when 
\beeq
\label{lo-mixed}
1<p<1+\frac{2}{n-2}, \ 1<q<1+\frac{4}{n-4}\ .
\eneq

\begin{prop}
Let $n\geq 5$ and $(p, q)$ in region \eqref{lo-mixed} and $c_1c_2\neq 0$. Consider \eqref{nlw2} on asymtotically Euclidean manifolds $\gm$ with $\gm=\gm_{1}$ satisfies \eqref{1029} and $\theta$ small enough. Then \eqref{nlw2} is locally well-posed in $H_{rad}^2\times H_{rad}^1$ for small enough $\ep$.
\end{prop}
\begin{proof}
The proof follows the similar way in Proposition \ref{lem-2.4}, while for $p_G<p<1+\frac{2}{n-2}$, we shall make a slightly modification of $\mu$, that is, we set $\mu=\frac{(n-2)(p-1)}{2}\in (0, 1)$ and in this case, we see that 
\begin{align*}
 \|\nabla^{\le 1}|u_t|^p\|_{N_T}\les &T^{(1-\mu)/2}
\|r^{\mu/2}\nabla^{\le 1}|u_t|^p\|_{L^2_{T,x}}\\
\les &T^{1-\mu} \|r^{(n-2)/2} u_t\|_{L^\infty}^{p-1}
\|\nabla^{\le 1}u\|_{X_T}
\les T^{1-\mu} \|\nabla^{\le 1}u\|_{X_T}^{p}\ .
\end{align*}

 So we need only to show how to control the additional norm of $|u|^q$. With $\|\nabla^{\leq1} u\|_{X_T}\les \ep$, we claim that 
\beeq
\label{JY4}
\|\nabla^{\leq1}|u|^q\|_{N_T}\les\ \ep^{q}\ll \ep \ ,
\eneq
for some small $T$ and $\ep$. 

On the one hand, by H\"older's inequality, we have  
$$\|\nabla^{\leq1}|u|^q\|_{N_T}\leq \|\nabla^{\leq1}|u|^q\|_{L_{T}^1L_{x}^2}\les T(\|u\|^q_{L_{T}^\infty L^{2q}}+\||u|^{q-1}\nabla u\|_{L_{T}^\infty L^2})\ .$$
Notice that $\frac{n}{2}-\frac{1}{q-1}, \frac{n}{2}-\frac{n}{2q}\leq 2$ when $1<q\leq 1+\frac{2}{n-4}$, then by Sobolev embedding, we get 
$$\|u\|_{L_x^{2q}}\les \|u\|_{\dot{H}^{\frac{n}{2}-\frac{n}{2q}}}\les \|u\|_{H^2}, $$
$$
\||u|^{q-1}\nabla u\|_{L_x^2}\les \|\nabla u\|_{L^{\frac{2n}{n-2}}}\|u\|^{q-1}_{L^{n(q-1)}}\les\|u\|_{\dot{H}^2}\|u\|^{q-1}_{\dot{H}^{\frac{n}{2}-\frac{1}{q-1}}}\les \|u\|^q_{H^2}\ .$$
For the $L^2$ norm of $u$, by Newton-Leibniz formula, we have 
$$\|u\|_{L_{T}^{\infty}L_{x}^2}\leq T\|u_t\|_{L_{T}^{\infty}L_{x}^2}+\|u_0\|_{L^2}\leq T\|u\|_{X_T}+\ep\ .$$
 
 On the other hand, we also have 
  $$\|\nabla^{\leq 1} |u|^q\|_{N_T}\les T^{\frac{1-\mu}{2}}\|r^{\frac{\mu}{2}}|u|^{q-1}\nabla^{\leq 1}u\|_{L_{T,x}^{2}}\ .$$
By H\"older's inequality, we get 
$$\|r^{\frac{\mu}{2}}|u|^{q-1}\nabla^{\leq 1}u\|_{L_{T,x}^{2}}\les\|r^{\frac{\mu+1}{q-1}}u\|^{q-1}_{L_{t,x}^{\infty}}\|r^{-1-\frac{\mu}{2}}\nabla^{\leq 1}u\|_{L_{T,x}^{2}}\ .$$
Note that $\frac{n}{2}-\frac{\mu+1}{q-1}\in (\frac{1}{2}, 2]$ if $1+\frac{2}{n-1}<q<1+\frac{4}{n-4}$, then by applying trace estimates, we obtain that
$$\|r^{\frac{\mu+1}{q-1}}u\|_{L_x^\infty}\les \|u\|_{\dot{H}^{\frac{n}{2}-\frac{\mu+2}{q-1}}}\les \|u\|_{H^2}\ .$$
Thus we have
$$\|\nabla^{\leq 1} |u|^q\|_{N_T}\les \ T^{1-\mu}\|u\|^{q-1}_{L_{T}^{\infty}H^2}\ep\ \les \ \ep^{q}\ll\ep\ ,$$
if we take $T$ small enough.
\end{proof}

\section{Finite time blow up of \eqref{nlw2} with $c_1>0$}\label{section3}
In this section we give the proof of Theorem \ref{main-a-g},
in the case $c_1>0$ with estimates \eqref{main-es},
 under the assumption there is a local solution $u\in C^{2}([0, T_\ep);\mathcal{D}'(\R^{n}))$ with $u_{t}\in L_{loc,t,x}^{p}([0, T_{\ep})\times\R^{n})$ and $|u|^{q}\in C([0, T_\ep);\mathcal{D}'(\R^{n}))$ satisfies finite speed of propagation. And based on the local well-posedness in Section 2, we are ready to prove the finite time blows up in Theorem \ref{main-2}-Theorem \ref{main-5} without these assumption. 

By the transformation \eqref{bianhuan}, for future reference, we record that 
\beeq\label{eq-change}
\eta'(t)=m(\eta(t))=\tilde m(t),\ 
\tilde m'(t)=b(\eta(t))\tilde m^2(t),\ 
\eta(0)=0, m(0)=\tilde m(0)=1\ .
\eneq

\subsection{Elliptic equation $\Delta_{\gm}\phi=\lambda^{2}\phi$}
As we discussed in introduction, in our setting, \eqref{H1} is assumed to be true for some fixed $\la>0$ when $\gm=\gm_1+\gm_2$ and we have \eqref{H1} when $\gm=\gm_1+\gm_3$.  Then we will exploit it to construct the solution of linear wave equation 
\beeq
\label{lin-part}
\pa^2_t u-\tilde{m}^2\Delta_{\gm} u=0\ .
\eneq

\subsection{Test function $\psi(t, x)$}
As usual, we want to find a nonnegative solution $\psi(t, x)$ for \eqref{lin-part} by methods of separation of variables. Let $\psi(t, x)=\varphi(t)\phi(x)$. Then we are reduced to consider the following ordinary differential equation:
\beeq
\label{levi}
\varphi''-\la^2\tilde{m}^{2}(t)\varphi=0\ .
\eneq
By applying Levinson theorem (see, e.g., 
\cite[Chapter 3, Theorem 8.1]{CoLe55}), there exist a $T>0$ such that 
\beeq
\label{lev-proof}
\varphi\simeq  
e^{-\lambda\int^{t}_{t_0}\tilde{m}(\tau)d\tau}
\simeq   e^{-\lambda \eta(t)}
\ , \forall t\geq T\ ,
\eneq
\beeq
\label{lev-proof1}
\varphi'\simeq  
-\la k e^{-\lambda\int^{t}_{t_0}\tilde{m}(\tau)d\tau}
\simeq   -\la k e^{-\lambda \eta(t)}
\ , \forall t\geq T\ ,
\eneq
where
$k=\tilde{m}(\infty)=\exp(\int_0^\infty b(t)dt)$. In fact, let $Y(t)=(Y_{1}(t), Y_{2}(t))^{T}$ and $Y_{1}=\varphi$, $Y_{2}=\varphi'$, then we have 
$$Y'=(A+V(t))Y,
A=\left(\begin{array}{cc}0 & 1 \\ \lambda^{2}k^2 & 0\end{array}\right),
V(t)=\left(\begin{array}{cc}0 & 0 \\ \lambda^{2}(\tilde{m}^{2}(t)-k^2) & 0\end{array}\right)\ .$$
Noticing that $V'\in L^1$ and $\lim_{t\to\infty} V(t)=0$, we could apply the Levinson theorem to the system. Then there exists $t_0\in [0,\infty)$ so that we have a solution, which have the asymptotic form
when $t\geq T$,
\beeq
\label{K1}
Y(t)=
\left(\begin{array}{c}1+o(1) \\ 
-\lambda k+o(1)\end{array}\right)
e^{-\lambda\int^{t}_{t_0}\tilde{m}(\tau)d\tau},
\eneq
which gives us \eqref{lev-proof}-\eqref{lev-proof1}.

\begin{lem}
Let $\nu(t)=-\varphi'/\varphi$, then there exists $\delta_2\in (0, 1)$, such that for any $t\geq 0$, we have 
\beeq
\la\delta_2\leq\nu(t)\leq \la/\delta_2\  .
\eneq
\end{lem}
\begin{proof}
By \eqref{K1}, we have
$$\nu(t)=-\frac{\varphi'}{\varphi}\to \la k, \ t\to \infty,$$
and there exists a $N>T$, such that 
$$
\nu(t)\in [\frac{\la k}{2}, 2\la k], \ t\geq N\ .
$$
For the remaining interval
$[0, N]$,
with data $\varphi(N)>0$, $\varphi'(N)<0$ at $t=N$, we know from the
 equation \eqref{levi} that
 $$\varphi'(t)\le \varphi'(N)<0,\ 
 \varphi(t)\ge \varphi(N)>0, \ 
 \forall t\in [0, N]\ .
 $$
Since $\nu\in C^{1}[0, N]$,
it tells us that
$\nu(t)>0$ and so 
$\nu(t)\sim 1$ on the compact interval
$[0, N]$. This completes the proof.
\end{proof}

\subsection{Finite speed of propagation}
For linear wave equation $\pa^{2}_{t}u-\Delta_{\gm_1}u+b(t)u_{t}=0$ with the initial data \eqref{hs2} the support of solution $u$ is 
\beeq
\label{fspyl}
\supp u\subset \{(t, x); \int^{|x|}_{0}K(\tau)d\tau\leq t+R_{1}\},R_{1}\geq\int_{0}^{R_{0}}K(\tau)d\tau\ .
\eneq
As $\gm_2$ is short-range perturbation, which does not affect speed of propagation too much,
we still have \eqref{fspyl}, with possibly bigger $R_1$ to \eqref{nlw2}. Thus with the change of variable $t\to \eta(t)$, the support of solution $u$ of \eqref{bianhuan1} is
\beeq
\label{supp01}
\supp u\subset \{(t, x); \int^{|x|}_{0}K(\tau)d\tau\leq \eta(t)+R_{1}\}=D\ .
\eneq
Recall that $m(t)\in [\delta_{1}, 1/\delta_{1}]$, by \eqref{eq-change} we have 
\beeq
D\subset\{(t,x); |x|\leq \frac{t}{\delta_{0}\delta_{1}}+\frac{R_{1}}{\delta_{0}}\}=\tilde{D}\ .
\eneq

\begin{lem}
\label{lemma3.1}
Let $q\geq1$, then we have
$$\int_{D}\psi^{q}d v_{\gm} \les (t+1)^{n-1-\frac{n-1}{2}q}\ .$$
\end{lem}
\begin{proof}
We divide the region $D$ into two disjoint parts: $D=D_{1}\cup D_{2}$ where
$$D_{1}=\{(t, x); \int^{|x|}_{0}K(\tau)d\tau\leq \frac{\eta(t)+R_{1}}{2}\}\ .$$
For the region $D_{1}$, by \eqref{H1}, we have 
$$\int_{D_{1}}\psi^{q} d v_{\gm}\les e^{-\la\eta(t)q}\int_{D_{1}}(1+\la|x|)^{-\frac{n-1}{2}q}e^{q\la\int^{|x|}_{0}K(\tau)d\tau}d v_{\gm}\ .$$
Let $\tilde{r}=\int^{|x|}_{0}K(\tau)d\tau$, then $d\tilde{r}=K(r)dr$ and $\delta_{0}r\leq \tilde{r}\leq r/\delta_{0}$ since $K\in[\delta_{0}, 1/\delta_{0}]$. Then we get 
\begin{align*}
\int_{D_{1}}\psi^{q} d v_{\gm}&\les e^{-\la\eta(t)q}\int^{\frac{\eta(t)+R_{1}}{2}}_{0}(1+\tilde{r})^{n-1-\frac{n-1}{2}q}e^{q\la\tilde{r}}d\tilde{r}\\
&\les e^{-\la\eta(t)q}\int^{\frac{\eta(t)+R_{1}}{2}}_{0} e^{\frac{3}{2}q\la\tilde{r}}d\tilde{r}\\
&\les e^{-\frac{\la}{4}\eta(t)} \les (1+t)^{n-1-\frac{n-1}{2}q},
\end{align*}
where we have used the fact that $e^{-t}$ decays faster than any polynomial. 
For the region $D_{2}$, it is easy to see
\begin{align*}
\int_{D_{2}}\psi^{q} d v_{\gm}&\les e^{-\la\eta(t)q}\int^{\eta(t)+R_{1}}_{\frac{\eta(t)+R_{1}}{2}}e^{q\la\tilde{r}}(1+\tilde{r})^{n-1-\frac{n-1}{2}q}d\tilde{r}\\
&\les (1+\eta(t))^{n-1-\frac{n-1}{2}q}\int^{\eta(t)+R_{1}}_{\frac{\eta(t)+R_{1}}{2}}e^{q\la\tilde{r}-\la\eta(t)q}d\tilde{r}\\
&\les (1+t)^{n-1-\frac{n-1}{2}q}\ ,
\end{align*}
which completes the proof.
\end{proof}

\subsection{Proof of  \eqref{main-es}} For simplicity, we first define some auxiliary functions 
$$F(t)=\int_{\R^{n}}u_{t}(t, x)\psi(t, x)d v_{\gm}\ ,$$
$$G(t)=\int_{\R^{n}}u(t, x)\psi(t, x)d v_{\gm}\ ,$$
$$H(t)=c_1\int_{\R^{n}}\tilde{m}^{2-p}(t)|u_{t}|^{p}\psi(t, x)d v_{\gm}\ ,$$
which are well-defind in the setting of Theorem \ref{main-2}-\ref{main-5}.
  In the setting of Theorem \ref{main-a-g}, for fixed $t$, $u, u_{t}, |u_{t}|^{p}, |u|^q$ could be viewed as distribution with compact support, thanks to the finite speed of propagation. 
As a consequence, we see that $F, G\in C^1([0, T_\ep)$.
Since $c_1, c_2\ge 0$,
both $c_1 \tilde{m}^{2-p}(t)|u_{t}|^{p}$
 and $c_2 \tilde{m}^{2}(t)|u|^{q}$ are non-negative distribution with compact support,
which particularly gives us
$H(t)\ge 0$, as $\psi\ge 0$.

Using $\psi, \psi_t\in C^\infty$ as the test functions, 
we know that
$\<u_{t},\psi\>-\<u , \psi_{t}\>\in C^1$
and
 we get from  \eqref{bianhuan1}, $\psi\ge 0$ and
 $\pt^2\psi=\Delta_\gm \psi$
  that
\begin{align*}
\frac{d}{dt}\int \big(u_{t}\psi-u \psi_{t}\big) d v_{\gm} 
=&\int \big(u_{tt}\psi-u \psi_{tt}\big) d v_{\gm}\\
=&\int \big(u_{tt}\psi-\tilde{m}^2 u\Delta_{\gm}\psi \big) d v_{\gm}
=\int \big(u_{tt}-\tilde{m}^2\Delta_{\gm}u\big)  \psi d v_{\gm}\\
\ge &c_1\int \tilde{m}^{2-p}(t)|u_{t}|^{p}\psi d v_{\gm}\ ,
\end{align*}
and so
\beeq
\label{eq-920}
(F+\nu G)'\geq H .
\eneq
Then we get
\beeq
\label{920}
F+\nu G\geq F(0)+\nu(0)G(0)+\int^{t}_{0}H(\tau) d\tau \geq 0\ ,
\eneq
thanks to the assumption \eqref{hs2} on the initial data.
Notice that 
\begin{align*}
(\nu G)'=-\frac{d}{dt}\int u\psi_{t} d v_{\gm}=&-\Big(\int u_{t}\psi_{t} d v_{\gm}+\int u \psi_{tt} d v_{\gm}\Big)\\
=&\nu \int u_{t}\psi d v_{\gm}-\la^2\tilde{m}^{2}\int u \psi d v_{\gm}\\
=&\nu F-\la^{2}\tilde{m}^{2}G\ .
\end{align*}
We get also from \eqref{eq-920} that
\beeq
\label{920-1}
F'+\nu F-\la^{2}\tilde{m}^{2}G\geq H\ .
\eneq
Based on \eqref{920} and \eqref{920-1},
it is easy to see that $$F(t)\geq 0\ , G(t)\geq 0\ , \forall t\geq 0\ .$$
Actually, recall that $G'=\int u_t \psi +u\psi_t d v_{\gm}=F-\nu G$, we have
$$G'+2\nu G=F+\nu G\geq 0$$
due to \eqref{920}, which gives us $G\geq e^{-\int^{t}_{0}2\nu d\tau}G(0)\geq 0$. Then, by \eqref{920-1}, we see that $F'+\nu F\geq 0$, which yields $F\geq 0$. 

Next, by multiplying $B_1=\la\delta^{2}_{1}\delta_2$ to \eqref{920} and adding it with \eqref{920-1}, we obtain
$$ F'+(\nu+B_1)F\geq (\la^{2}\tilde{m}^{2}-B_1\nu)G+B_1F(0)+B_1\nu(0)G(0)+B_1\int^{t}_{0} H d\tau+H\ .$$
Note that $\nu\leq \la/\delta_{2}$, $\tilde{m}\geq \delta_{1}$ then $\la^{2}\tilde{m}^{2}-B_1\nu\geq0$ and 
\beeq
\label{BSD}
F'+(\frac{\la}{\delta_2}+B_1)F\geq B_1F(0)+B_1\nu(0)G(0)+B_1\int^{t}_{0} H d\tau+H\ .
\eneq
Let $B_2=B_1/\big((\la/\delta_2) +B_1\big)<1$,
$L(t):=F(t)-I(t)$ with increasing function
\beeq
I(t):=B_2\int^{t}_{0} H d\tau+B_2F(0)+\frac{B_2}{2}\nu(0)G(0)\ ,\eneq
then $L'=F'-B_2H$, $L(0)=(1-B_2)F(0)-\frac{B_2}{2}\nu(0)G(0)\geq -\frac{B_2}{2}\nu(0)G(0)$ and 
\begin{align*}L'+\frac{B_1}{B_2}L=&F'+\frac{B_1}{B_2}F-B_2H-B_1\int^{t}_{0} H d\tau-B_1F(0)-\frac{B_1}{2}\nu(0)G(0)\\
\geq& (1-B_2)H+\frac{B_1}{2}\nu(0)G(0)\geq \frac{B_1}{2}\nu(0)G(0)\ ,
\end{align*}
thanks to \eqref{BSD}. Thus we have 
\begin{align*}
L(t)&\geq e^{-\frac{B_1}{B_2}t}L(0)+e^{-\frac{B_1}{B_2}t}\frac{B_1}{2}\nu(0)G(0)\int^{t}_{0}e^{\frac{B_1}{B_2}\tau}d\tau\\
&\geq \frac{B_2}{2}\nu(0)G(0)\big(1-2e^{-\frac{B_1}{B_2}t}\big)\geq 0,
\end{align*}
when $t\geq t_{0}=\frac{B_2}{B_1}\ln 2$, which yields 
$$I(t)\leq F(t)=\int u_{t}\psi d v_{\gm}\ ,\  \forall t\geq t_0\ .$$
By H\"older's inequality and Lemma \ref{lemma3.1} with $q=1$, we get
$$I(t)\les \Big(\int |u_{t}|^{p}\psi d v_{\gm}\Big)^{1/p}\Big(\int_{D} \psi d v_{\gm}\Big)^{1/p'}
\les
H^{1/p}
(t+1)^{\frac{n-1}{2p'}}
, \ \forall t\geq t_0\ ,$$
which gives us the following ordinary differential inequality  for $I(t)$
\beeq\label{eq-bu1}
I'(t) =B_2H(t)\gtrsim \frac{I^{p}}{(1+t)^{\frac{n-1}{2}(p-1)}},\ \forall t\geq t_0\ .\eneq
On the other hand,
as $I'(t)=B_2 H(t)\geq 0$ and the initial data $(u_0, u_1)$ are nontrivial
with \eqref{hs2},
we know that 
\beeq\label{eq-bu2} I(t_0)\geq I(0)=B_2\ep\int u_1\phi d v_{\gm}+\frac{B_2}{2}\nu(0)\ep\int u_0\phi d v_{\gm} \sim \ep\ .\eneq
 
With help of \eqref{eq-bu1} and
\eqref{eq-bu2}, it is easy to conclude the proof of  Theorem \ref{main-a-g} for the case
$c_1>0$ with estimates \eqref{main-es}.
Actually,
when $\frac{n-1}{2}(p-1)=1$, that is, $p=p_G$, we see that for some positive constants $C$, $\ep_0$, such that for any $\ep\in(0, \ep_0)$, we have 
$$I(t)\geq \Big(C\ep^{1-p}-(p-1)\ln(1+t)\Big)^{-1/(p-1)},$$
 which will go to $\infty$ as $C\ep^{1-p}-(p-1)\ln(1+t)\to 0$. Thus we have the existence time $T_\ep$ of $I(t)$ satisfies
$$T_\ep \leq \exp(\tilde{C}\ep^{-(p-1)})\ ,$$
for some $\tilde{C}>0$. Similarly, when $\frac{n-1}{2}(p-1)<1$, that is, $p<p_G$, we get
$$I(t)\geq \Big(C\ep^{1-p}-
\frac{p-1}{1-(n-1)(p-1)/2}
(1+t)^{1-(n-1)(p-1)/2}\Big)^{-1/(p-1)}\ .$$
Hence the existence time $T_\ep$ satisfies
$$T_\ep \leq C_0\ep^{-\frac{2(p-1)}{2-(n-1)(p-1)}},$$
for some $C_0>0$ and small enough $\ep$.

\section{Finite time blow up of \eqref{nlw2} with $c_1c_2\neq0$}
In this section, we present a proof of \eqref{mixed-lifespan1},
for \eqref{nlw2} with $c_1c_2\neq0$.
For that purpose, we introduce the function
$$F(t)=\int u(t, x) d v_{\gm}\ ,$$
which, as we shall see, will blow up in finite time. 

In view of \eqref{bianhuan1}, we have 
\beeq
\label{P-11}
F''=\int c_1\tilde{m}^{2-p}(t)|u_{t}|^{p}+c_2\tilde{m}^{2}(t)|u|^{q} d v_{\gm}\gtrsim\int |u|^{q} d v_{\gm}.
\eneq
By H\"older's inequality, we have 
$$|F|\leq \int_{\tilde{D}}|u| dv_{\gm}\les \Big(\int_{\tilde{D}}|u|^{q} dv_{\gm}\Big)^{1/q}(1+t)^{n/q'},$$
which,
combined with \eqref{P-11},
gives us
$$
F''\gtrsim |F|^{q}(1+t)^{-n(q-1)}\ .
$$
Since $F''\geq 0$, we have $F'(t)\geq F'(0)\geq 0 $ and $F(t)\geq F(0)+tF'(0)\geq 0$. Then we get
\beeq
\label{ODE-1}
F''\gtrsim |F|^{q}(1+t)^{-n(q-1)}\ , F(t)\geq F(0)+tF'(0)\geq 0, \ \forall t\geq 0\ .
\eneq

To show blow up, we 
need to
get better lower bound of $F$, for which purpose, we use the test function $$\tilde{\psi}(t, x)=e^{-\la\eta(t)}\phi(x)$$ with $\phi$ satisfying \eqref{H1},
and introduce the auxiliary function
$$G(t)=\int u_{t}\tilde{\psi} d v_{\gm}\ .$$
Without loss of generality, we suppose $T_{\ep}> 1$.
Then we claim that 
\beeq
\label{P-2}
G(t)\geq C\ep, \forall t\in [1, T_\ep),
\eneq
for some constant $C>0$.
With \eqref{P-2} in hand, we see that by H\"older's inequality
\beeq
\label{gaijin3}
C\ep\leq G(t)\les \Big(\int_{D}|u_{t}|^{p} d v_{\gm}\Big)^{1/p}\Big(\int_{D}\tilde{\psi}^{p'} d v_{\gm}\Big)^{1/p'},
\eneq
Recall that by Lemma \ref{lemma3.1}, we have
$$\int_{D}\tilde{\psi}^{p'} d v_{\gm} \les (1+t)^{n-1-\frac{n-1}{2}p'},
$$
which, recalling
 \eqref{P-11},
 yields 
\beeq
\label{gaijin1}
F''\gtrsim
\int |u_t|^pdv_\gm
 \gtrsim \ep^{p}(1+t)^{-(n-1)(p-2)/2}, t\geq 1\ .
\eneq

Notice that, if  $(n-1)(p-2)/2\leq 1$, that is $p\leq \frac{2n}{n-1}$ in the condition \eqref{mixed-region1}, then heuristically, we can integrate the above expression twice and obtain
\beeq
\label{ODE-2}
F \gtrsim \ep^{p}(1+t)^{2-(n-1)(p-2)/2}\ , t\geq 1\ ,
\eneq
which could be improved with the help of \eqref{ODE-1}, if 
$$\Big(2-\frac{(n-1)(p-2)}{2}\Big)q-n(q-1)+2>2-\frac{(n-1)(p-2)}{2}\ .$$
This is exactly the first condition of \eqref{mixed-region1} on $p,q$. 
Then the classical Kato type lemma (see, e.g., \cite[Lemma 2.1]{HZ2014})
could be applied to prove blow up results and 
the desired upper bound of lifespan \eqref{mixed-lifespan1}.

 
It remains to prove the claim \eqref{P-2}. Let $H(t)=\int u(t, x)\phi(x) d v_{\gm}$, then 
\beeq
\label{11-2}
G(t)=e^{-\la\eta(t)}H'(t),
\eneq
and $H$ satisfies
$$H''=\int u_{tt}\phi d v_{\gm}\geq \tilde{m}^2\int \Delta_{\gm}u  \phi d v_{\gm}\geq \tilde{m}^2\la^2H\ .$$
Let $y$ be the solution of ordinary differential equation
$$y''-\la^2\tilde{m}^2y=0, \ y(0)=H(0)/\ep=C_0 , y'(0)=H'(0)/\ep=C_1 \ ,$$
for some $C_{0}, C_{1}\ge 0$ with $C_{0}+C_{1}>0$. We observe that $H'\geq \ep y'$. 

For $y$, it is clear from continuity that we have
$y,  y'> 0$ for all $t> 0$.
More precisely, we have
$y(t)\ge C_{0}$, $y''\ge \de_1^2\la^2 C_{0}$ for all $t\ge 0$,
and so \beeq\label{eq-lb4}
y'\ge C_{1}+t\de_1^2\la^2 C_{0}
>0\ ,\ \forall t> 0\ .\eneq
Let $Y(t)=(Y_{1}(t), Y_{2}(t))^{T}$ and $Y_{1}=y$, $Y_{2}=y'$, then we have 
$$Y'=(A+V(t))Y,
A=\left(\begin{array}{cc}0 & 1 \\ \lambda^{2}k^2 & 0\end{array}\right),
V(t)=\left(\begin{array}{cc}0 & 0 \\ \lambda^{2}(\tilde{m}^{2}(t)-k^2) & 0\end{array}\right),$$
where
$k=\tilde{m}(\infty)=\exp(\int_0^\infty b(t)dt)$.
Noticing that $V'\in L^1$ and $\lim_{t\to\infty} V(t)=0$, we could apply the Levinson theorem to the system. Then there exists $t_0\in [0,\infty)$ so that we have two independent solutions, which have the asymptotic form
as $t\to \infty$
$$Y_{\pm}(t)=
\left(\begin{array}{c}1+o(1) \\ 
\pm\lambda k+o(1)\end{array}\right)
e^{\pm\lambda\int^{t}_{t_0}\tilde{m}(\tau)d\tau}
$$
Thus, we have, for some $d_1, d_2$,
\beeq
y'=d_1(\la k+o(1))
e^{\lambda\int^{t}_{t_0}\tilde{m}(\tau)d\tau}
+d_2(-\la k+o(1))
e^{-\lambda\int^{t}_{t_0}\tilde{m}(\tau)d\tau}\eneq
as $t\to\infty$.
By \eqref{eq-lb4}, we know that
$y'\ge C_{1}+\de_1^2\la^2 C_{0}$
 for all $t\ge 1$, it is clear that $d_1>0$.
Then there exists some $T>1$ such that
\beeq
y'\simeq  d_1 
e^{\lambda\int^{t}_{t_0}\tilde{m}(\tau)d\tau}
\simeq  d_1 e^{\lambda \eta(t)}
\ ,\ \forall t\ge T\ . \eneq
So we have 
$$y'\geq ce^{\lambda \eta(t)}, \forall t\geq 1,$$
for some $C>0$.

In conclusion, 
we obtain
$$H'\ge \ep y'\ge c\ep e^{\lambda \eta(t)}\ , \forall t\ge 1\ .$$
which completes the proof of
 \eqref{P-2}.

\bibliographystyle{plain1}

\end{document}